\documentclass[12pt,a4paper,leqno]{article}

\usepackage[nottoc]{tocbibind}
\usepackage{latexsym}
\usepackage{amssymb,exscale}
\usepackage[centertags]{amsmath}
\usepackage{amsthm}
\usepackage{footnpag}

\usepackage[linktocpage]{hyperref}

\numberwithin{equation}{subsection}

\swapnumbers
\theoremstyle{definition}
\newtheorem{theorem}[equation]{Theorem}
\newtheorem{lemma}[equation]{Lemma}
\newtheorem{corollary}[equation]{Corollary}
\newtheorem{definition}[equation]{Definition}
\newtheorem{proposition}[equation]{Proposition}
\newtheorem{remark}[equation]{Remark}
\newtheorem{assumption}[equation]{Assumption}

\renewcommand{\phi}{\varphi}

\newcommand{\D}{\mathrm{d}}

\newcommand{\ti}{\tilde}

\renewcommand{\(}{\bigl(}
\renewcommand{\)}{\bigr)\vphantom{)}}

\newcommand{\ip}[2]{\langle#1,#2\rangle}

\renewcommand{\equiv}{\;\>\Longleftrightarrow\;\>}

\newcommand{\Reg}{\operatorname{Reg}}
\newcommand{\Int}{\operatorname{Int}}
\newcommand{\Cl}{\operatorname{Cl}}
\newcommand{\Bd}{\operatorname{Bd}}

\newcommand{\One}{{1\hskip-2.5pt{\rm l}}}

\newcommand{\eps}{\varepsilon}
\newcommand{\Si}{\Sigma}

\newcommand{\Om}{\Omega}

\newcommand{\al}{\alpha}

\newcommand{\M}{\mathcal M}

\newcommand{\Ec}{\mathcal E}
\newcommand{\F}{\mathcal F}
\newcommand{\A}{\mathcal A}

\newcommand{\La}{\Lambda}

\newcommand{\Ex}{\mathbb E\,}
\newcommand{\R}{\mathbb R}

\newcommand{\cE}[2]{\mathbb{E}\mskip1.5mu\(\mskip1.5mu#1\mskip1.5mu
 \big|\mskip1.5mu#2\mskip1.5mu\)}

\newcommand{\sif}{$\sigma$\nobreakdash-field}

\newcommand{\atomless}[1]{$#1$\nobreakdash-\hspace{0pt}atomless}
\newcommand{\independent}[1]{$#1$\nobreakdash-\hspace{0pt}independent}
\newcommand{\measurable}[1]{$#1$\nobreakdash-\hspace{0pt}measurable}

\begin{document}

\title{Noise as a Boolean algebra of \sif s. II. Classicality,
  blackness, spectrum}

\author{Boris Tsirelson}

\date{}
\maketitle

\begin{abstract}
Similarly to noises, Boolean algebras of \sif s \cite{Ts} can be
black. A noise may be treated as a homomorphism from a Boolean algebra
of regular open sets to a Boolean algebra of \sif s. Spectral sets are
useful also in this framework.
\end{abstract}

\setcounter{tocdepth}{2}
\tableofcontents

\section*{Introduction}
\addcontentsline{toc}{section}{Introduction}
A noise is called \emph{black,} if its classical part is trivial (but
the whole noise is not) \cite[Def.~7a1]{Ts04}. The same definition
applies to noise-type Boolean algebras (of \sif s) introduced
in \cite{Ts}. Triviality of the classical part is treated here as
absence of non-zero square-integrable random variables $ \psi $
satisfying the following additivity condition:
\[
\psi = \cE{ \psi }{ \Ec } + \cE{ \psi }{ \Ec' }
\]
for all \sif s $ \Ec $ of the given Boolean algebra. The set of all
such $ \psi $ is the so-called first chaos space (generalizing the
first Wiener chaos space). Surprisingly, it is sufficient to check the
additivity condition only for \sif s $ \Ec $ of a Boolean subalgebra,
provided that the corresponding measure $ \Ec \mapsto \Ex
| \cE{\psi}{\Ec} |^2 $ on the subalgebra is atomless (Theorem
\ref{1a4}). This result is useful when dealing with a noise over
$ \R^2 $ that is not rotation-invariant. (Probably, the Arratia
flow leads to such noise.) Its projections to different axes may
behave quite differently, but anyway, if one of them is black then
others must be black.

The spectral theory of noises \cite[Sect.~9]{Ts04} is reformulated
here (Sect.~\ref{sec:2}) for a noise-type Boolean algebra. Instead of
spectral measures on the space of closed sets we get spectral measure
spaces. Sect.~\ref{sec:3} relates the new framework to the old one. If
$ \R^2 $ is divided in two domains by a curve, a noise over $ \R^2 $
is thus divided in two independent components if and only if almost
all spectral sets avoid the curve (Prop.~\ref{3b9}).

\section[Classicality and blackness]{\raggedright Classicality and blackness}
\label{sec:1}
\subsection{Definitions; preservation under completion}
\label{1a}

Let $ B $ be a noise-type Boolean algebra \cite[Def.~2a1]{Ts} of \sif
s on a probability space $ (\Om,\F,P) $. The corresponding
projections\footnote{%
 Throughout, ``projection'' means ``orthogonal projection''.}
$ Q_x $ \cite[Sect.~1d]{Ts} for $ x \in B $, acting on $ H =
L_2(\Om,\F,P) $, satisfy \cite[Lemma~2a2]{Ts}
\begin{equation}\label{1a1}
Q_x Q_y = Q_{x \wedge y} \, .
\end{equation}

\begin{definition}
(a) The \emph{first chaos space} $ H^{(1)} $ is a (closed linear)
subspace of $ H $ consisting of all $ \psi \in H $ such that for all $
x,y \in B $
\begin{equation}\label{1a3}
x \wedge y = 0 \quad \text{implies} \quad Q_{x\vee y} \psi = Q_x \psi
+ Q_y \psi \, .
\end{equation}
(b) $ B $ is called \emph{classical} if the first chaos space
generates the whole \sif\ $ \F $.
\newline
(c) $ B $ is called \emph{black} if the first chaos space contains
only $ 0 $.
\end{definition}

Taking $ x=y=0 $ in \eqref{1a3} we see that
\begin{equation}\label{*}
Q_0 \psi = 0 \quad \text{for all } \psi \in H^{(1)} \, .
\end{equation}

Note that $ H^{(1)} $ is the set of all $ \psi \in H $ such that $ Q_0
\psi = 0 $ and for all $ x,y \in B $
\begin{equation}\label{1a35}
Q_{x\vee y} \psi + Q_{x\wedge y} \psi = Q_x \psi + Q_y \psi \, ;
\end{equation}
for the proof, apply \eqref{1a3} twice: to $ x, y \wedge x' $ and also
to $ y \wedge x, y \wedge x' $.

Recall the noise-type completion $ C $ of $ B $ \cite{Ts}.

\begin{proposition}
If $ \psi $ satisfies \eqref{1a3} for all $ x,y \in B $ then $ \psi $
satisfies \eqref{1a3} for all $ x,y \in C $, where $ C $ is the
noise-type completion of $ B $.
\end{proposition}

\begin{proof}
We use \eqref{1a35} instead of \eqref{1a3}. For every $ x \in C $ the
maps $ y \mapsto x \wedge y $ and $ y \mapsto x \vee y $ are
continuous on $ C $ \cite[(2a6) and 2b6]{Ts} in the ``strong
operator'' topology \cite[Sect.~1d]{Ts}: $ x_n \to x $ in this
metrizable topology if and only if $ \forall \psi \in H \,\> \|
Q_{x_n} \psi - Q_x \psi \| \to 0 $.

Thus, \eqref{1a35} extends by continuity from the case $ x,y \in B $
to the more general case $ x \in B $, $ y \in C $. And then it extends
further to $ x,y \in C $.
\end{proof}

We see that $ H^{(1)} $ (as well as classicality and blackness) is
uniquely determined by the completion $ C $ of $ B $. Recall also that
$ C $ is uniquely determined by the closure of $ B $
\cite[Intro]{Ts}.

\subsection{Beyond the completion}
\label{1b}

A partition of unity in $ B $ consists, by definition, of $
x_1,\dots,x_n \in B $ such that $ x_1 \vee \dots \vee x_n = 1 $, $ x_i
\ne 0 $ for all $ i $, and $ x_i \wedge x_j = 0 $ whenever $ i \ne j
$.

We say that a vector $ \psi \in H^{(1)} $ is atomless, if for every $
\eps > 0 $ there exists a partition of unity $ x_1,\dots,x_n $ such
that $ \| Q_{x_i} \psi \| \le \eps $ for all $ i = 1,\dots,n $.

Assume that $ b \subset B $ is a Boolean subalgebra, and a vector $
\psi \in H $ satisfies \eqref{1a3} for all $ x,y \in b $. The notion
``\atomless{b}'' is defined as before (using partitions of unity in $
b $ rather than $ B $).

\begin{theorem}\label{1a4}
If $ b $ is a Boolean subalgebra of $ B $, $ \psi \in H $ satisfies
\eqref{1a3} for all $ x,y \in b $ and is \atomless{b}, then $ \psi
\in H^{(1)} $.
\end{theorem}

The proof is given after some lemmas.

Note that
\[
\ip{ Q_x \psi }{ Q_y \psi } = 0 \quad \text{whenever } \psi \in
H^{(1)} \text{ and } x \wedge y = 0 \, ,
\]
since $ \ip{ Q_x \psi }{ Q_y \psi } = \ip{ Q_y Q_x \psi }{ \psi } =
\ip{ Q_0 \psi }{ \psi } = 0 $ by \eqref{1a1} and \eqref{*}. It follows
that
\[
x \mapsto \| Q_x \psi \|^2 \quad \text{is an additive function } B \to
[0,\infty) \quad \text{for } \psi \in H^{(1)} \, .
\]

\begin{lemma}
$ Q_x + Q_y \le Q_{x\vee y} + Q_{x\wedge y} $ for all $ x,y \in B $.
\end{lemma}

\begin{proof}
By \eqref{1a1}, $ Q_x $ and $ Q_y $ are commuting projections, which
implies $ Q_x + Q_y = Q_x \vee Q_y + Q_x \wedge Q_y $, where
$ Q_x \vee Q_y $ and $ Q_x \wedge Q_y $ are projections onto $ Q_x H +
Q_y H $ and $ Q_x H \cap Q_y H $ respectively. Using \eqref{1a1}
again, $ Q_x \wedge Q_y = Q_x Q_y = Q_{x\wedge y} $. It remains to
note that $ Q_x \vee Q_y \le Q_{x\vee y} $ just because $ Q_x \le
Q_{x\vee y} $ and $ Q_y \le Q_{x\vee y} $.
\end{proof}

Taking into account that $ \| Q_x \psi \|^2 = \ip{ Q_x \psi }{ \psi }
$ we get the following.

\begin{corollary}\label{1a7}
For every $ \psi \in H $ such that $ Q_0 \psi = 0 $ we have
\[
x \mapsto \| Q_x \psi \|^2 \quad \text{is a superadditive function } B
\to [0,\infty) \, ,
\]
that is, $ \| Q_x \psi \|^2 + \| Q_y \psi \|^2 \le \| Q_{x\vee y} \psi
\|^2 $ whenever $ x \wedge y = 0 $.
\end{corollary}

\begin{lemma}
If $ \psi = Q_x \psi + Q_{x'} \psi $ for all $ x \in B $, then $ \psi
\in H^{(1)} $.
\end{lemma}

(Here $ x' $ is the complement of $ x $ in $ B $, of course.)

\begin{sloppypar}
\begin{proof}
If $ x \wedge y = 0 $ then $ Q_{x\vee y} \psi = Q_{x\vee y} ( Q_x \psi
+ Q_{x'} \psi ) = Q_{x\vee y} Q_x \psi + Q_{x\vee y} Q_{x'} \psi =
Q_{(x\vee y)\wedge x} \psi + Q_{(x\vee y)\wedge x'} \psi = Q_x \psi +
Q_y \psi $.
\end{proof}
\end{sloppypar}

Recall the sub-\sif s $ \F_x \subset \F $ and subspaces $ H_x =
L_2(\F_x) \subset H $ for $ x \in B $ \cite[Sect.~1a]{Ts}; $ H_x = Q_x
H $.

For every $ x \in B $ the \sif s $ \F_x, \F_{x'} $ are independent
\cite[Sect.~2a]{Ts}, therefore the pointwise product $ \xi \eta $
belongs to $ H $ for all $ \xi \in H_x $, $ \eta \in H_{x'} $.

\begin{lemma}\label{1a9}
The following two conditions on $ x \in B $ and $ \psi \in H $ are
equivalent:

(a) $ \psi = Q_x \psi + Q_{x'} \psi $;

(b) $ \Ex \psi = 0 $, and $ \Ex ( \psi \xi \eta ) = 0 $ for all $ \xi
\in H_x $, $ \eta \in H_{x'} $ satisfying $ \Ex \xi = 0 $,
$ \Ex \eta = 0 $.
\end{lemma}

(Here $ \Ex \psi = \int_{\Om} \psi \, \D P = \ip{ \psi }{ \One } $.)

The proof uses a construction important to \cite{Ts} (see first of all
\cite[proof of Prop.~1d13]{Ts}). Let $ x \in B $. Up to the natural
unitary equivalence we have $ H = H_x \otimes H_{x'} $ and $ Q_{u\vee
  v} = Q_u^{(x)} \otimes Q_v^{(x')} $ for all $ u,v \in B $ such that
$ u \le x $ and $ v \le x' $. Here $ Q_u^{(x)} : H_x \to H_x $ is the
projection onto $ H_u \subset H_x $; similarly, $ Q_v^{(x')} : H_{x'}
\to H_{x'} $ is the projection onto $ H_v \subset H_{x'} $. In
particular, $ Q_x = Q_x^{(x)} \otimes Q_0^{(x')} = \One \otimes
Q_0^{(x')} $ and $ Q_y = Q_0^{(x)} \otimes \One $.

It may be puzzling that $ H_x $ is both a subspace of $ H $ and a
tensor factor of $ H $ (which never happens in the general theory of
Hilbert spaces). Here is an explanation. All spaces $ H_x $ contain
the one-dimensional space $ H_0 $ of constant functions (on $ \Om
$). Multiplying an \measurable{\F_x} function $ \psi \in H_x $ by the
constant function $ \xi \in H_{x'} $, $ \xi(\cdot)=1 $, we get the
(puzzling) equality $ \psi \otimes \xi = \psi $.

\begin{proof}[Proof of Lemma \ref{1a9}]
Treating $ H $ as $ H_x \otimes H_{x'} $ we have $ H = \( (H_x
\ominus H_0) \oplus H_0 \) \otimes \( (H_{x'} \ominus H_0) \oplus H_0
\) = (H_x \ominus H_0) \otimes (H_{x'} \ominus H_0) \oplus (H_x
\ominus H_0) \otimes H_0 \oplus H_0 \otimes (H_{x'} \ominus H_0)
\oplus H_0 \otimes H_0 $; here $ H_x \ominus H_0 $ is the orthogonal
complement of $ H_0 $ in $ H_x $ (it consists of all zero-mean
functions of $ H_x $). In this notation $ Q_x + Q_{x'} $ becomes $
\One \otimes Q_0^{(x')} + Q_0^{(x)} \otimes \One = \( (\One-Q_0^{(x)})
+ Q_0^{(x)} \) \otimes Q_0^{(x')} + Q_0^{(x)} \otimes \(
(\One-Q_0^{(x')}) + Q_0^{(x')} \) = (\One-Q_0^{(x)}) \otimes
Q_0^{(x')} + Q_0^{(x)} \otimes (\One-Q_0^{(x')}) + 2 Q_0^{(x)} \otimes
Q_0^{(x')} $, the projection onto $ (H_x \ominus H_0) \otimes H_0
\oplus H_0 \otimes (H_{x'} \ominus H_0) $ plus twice the projection
onto $ H_0 \otimes H_0 $ ($=H_0$). Thus, the equality $ \psi =
(Q_x + Q_{x'}) \psi $ (Item (a)) becomes $ \psi \in (H_x \ominus H_0)
\otimes H_0 \oplus H_0 \otimes (H_{x'} \ominus H_0) $, or
equivalently, orthogonality of $ \psi $ to $ H_0 $ and $ (H_x
\ominus H_0) \otimes (H_{x'} \ominus H_0) $, which is Item (b).
\end{proof}

\begin{remark}\label{1a10}
The proof given above shows also that
\[
\{ \psi : \psi = Q_x \psi + Q_{x'} \psi \} = ( H_x \ominus H_0 )
\oplus ( H_{x'} \ominus H_0 )
\]
for all $ x \in B $.
\end{remark}

\begin{proof}[Proof of Theorem \ref{1a4}]
Let $ x \in B $; we have to prove that $ \psi = Q_x \psi + Q_{x'} \psi
$. Let $ \xi \in H_x \ominus H_0 $, $ \eta \in H_{x'} \ominus H_0 $;
by Lemma \ref{1a9} it is sufficient to prove that $ \Ex ( \psi \xi
\eta ) = 0 $.

Given $ \eps > 0 $, we take a partition of unity $ y_1,\dots,y_n $ in
$ b $ such that $ \| Q_{y_i} \psi \| \le \eps $ for all $ i $. We have
$ \psi = \sum_i Q_{y_i} \psi $ (by \eqref{1a3} for $ b $), thus, $ \Ex
( \psi \xi \eta ) = \sum_i \Ex \( (Q_{y_i} \psi) \xi \eta \) $.
Further, $ \Ex \( (Q_{y_i} \psi) \xi \eta \) = \ip{ Q_{y_i} \psi }{
\xi \otimes \eta } = \ip{ Q_{y_i} \psi }{ Q_{y_i} (\xi \otimes \eta) }
= \ip{ Q_{y_i} \psi }{ (Q^{(x)}_{u_i} \otimes Q^{(x')}_{v_i}) (\xi
\otimes \eta) } = \ip{ Q_{y_i} \psi }{ (Q^{(x)}_{u_i} \xi) \otimes
(Q^{(x')}_{v_i}) \eta) } $, where $ u_i = y_i \wedge x $ and $ v_i =
y_i \wedge x' $; it follows that $ | \Ex ( \psi \xi \eta ) | \le
\sum_i \| Q_{y_i} \psi \| \cdot \| Q^{(x)}_{u_i} \xi \| \cdot \|
Q^{(x')}_{v_i} \eta \| $. By additivity, $ \sum_i \| Q_{y_i} \psi
\|^2 = \| \psi \|^2 $. By superadditivity (Corollary \ref{1a7}), $
\sum_i \| Q^{(x)}_{u_i} \xi \|^2 \le \| \xi \|^2 $ and $ \sum_i \|
Q^{(x')}_{v_i} \eta \|^2 \le \| \eta \|^2 $. We get $ | \Ex ( \psi
\xi \eta ) | \le \( \max_i \| Q_{y_i} \psi \| \) \( \sum_i \|
Q^{(x)}_{u_i} \xi \| \cdot \| Q^{(x')}_{v_i} \eta \| \) \le \eps
\|\xi\| \|\eta\| $ for all $ \eps $.
\end{proof}

\section[Spectrum]{\raggedright Spectrum}
\label{sec:2}
\subsection{Preliminaries: commutative von Neumann algebras and
 measure class spaces}
\label{2a}

Every commutative von Neumann algebra $ \A $ of operators on a
separable Hilbert space $ H $ is isomorphic to the algebra $ L_\infty
(S,\Si,\mu) $ on some measure space $ (S,\Si,\mu) $
(\cite[Sect.~1.7.3]{Di}, \cite[Th.~1.22]{Ta}). Here and henceforth all
measures are positive, finite and such that the corresponding $ L_2 $
spaces are separable. The measure $ \mu $ may be replaced with any
equivalent (that is, mutually absolutely continuous) measure $ \mu_1
$. Thus we may turn to a measure class space (see
\cite[Sect.~14.4]{Ar}) $ (S,\Si,\M) $ where $ \M $ is an equivalence
class of measures, and write $ L_\infty(S,\Si,\M) $; we have an
isomorphism $ \al : \A \to L_\infty(S,\Si,\M) $ of von Neumann
algebras. (See \cite[14.4]{Ar} for the Hilbert space $ L_2(S,\Si,\M) $
on which $ L_\infty(S,\Si,\M) $ acts by multiplication.)

Let $ \Si_1 \subset \Si $ be a sub-\sif. Restrictions $ \mu|_{\Si_1} $
of measures $ \mu \in \M $ are mutually equivalent; denoting their
equivalence class by $ \M|_{\Si_1} $ we get a measure class space $
(S,\Si_1,\M|_{\Si_1}) $. Clearly, $ L_\infty(S,\Si_1,\M|_{\Si_1})
\subset L_\infty (S,\Si,\mu) $ or, in shorter notation, $
L_\infty(\Si_1) \subset L_\infty(\Si) $. We have $ L_\infty(\Si_1) =
\al(\A_1) $ where $ \A_1 = \al^{-1} ( L_\infty(\Si_1) ) \subset \A $
is a von Neumann algebra. And conversely, if $ \A_1 \subset \A $ is a
von Neumann algebra then $ \al(\A_1) = L_\infty(\Si_1) $ for some
sub-\sif\ $ \Si_1 \subset \Si $ (which follows easily from
\cite{Ba}).

Given two von Neumann algebras $ \A_1, \A_2 \subset \A $, we denote by
$ \A_1 \vee \A_2 $ the von Neumann algebra generated by $ \A_1, \A_2
$. Similarly, for two \sif s $ \Si_1, \Si_2 \subset \Si $ we denote by
$ \Si_1 \vee \Si_2 $ the \sif\ generated by $ \Si_1, \Si_2 $.

\begin{lemma}\label{2a13}
$ L_\infty(\Si_1) \vee L_\infty(\Si_2) = L_\infty(\Si_1 \vee \Si_2) $.
\end{lemma}

\begin{proof}
``$\subset$'' is trivial; we prove ``$\supset$''. Let $ A \in
L_\infty(\Si_1 \vee \Si_2) $; we have to prove that $ A \in
L_\infty(\Si_1) \vee L_\infty(\Si_2) $. Without loss of generality we
assume the following. First, that $ A $ is an indicator, $ A = \One_X
$, $ X \in \Si_1 \vee \Si_2 $. Second, that $ X $ belongs to the
\emph{algebra} generated by $ \Si_1, \Si_2 $. Third, that $ X =
X_1 \cap X_2 $ for some $ X_1 \in \Si_1 $, $ X_2 \in \Si_2 $. Now, $
\One_X = \One_{X_1} \One_{X_2} \in L_\infty(\Si_1) \vee
L_\infty(\Si_2) $.
\end{proof}

\begin{corollary}\label{2a15}
If $ \al(\A_1) = L_\infty(\Si_1) $ and $ \al(\A_2) = L_\infty(\Si_2) $
then $ \al(\A_1 \vee \A_2) = L_\infty(\Si_1 \vee \Si_2) $.
\end{corollary}

\begin{proof}
$ \al(\A_1 \vee \A_2) = \al(\A_1) \vee \al(\A_2) $, since $ \al $ is
an isomorphism; use \ref{2a13}.
\end{proof}

\begin{sloppypar}
The product $ (S,\Si,\mu) = (S_1,\Si_1,\mu_1) \times (S_2,\Si_2,\mu_2)
$ of two measure spaces leads to the tensor product of commutative von
Neumann algebras, $ L_\infty(S,\Si,\mu) = L_\infty(S_1,\Si_1,\mu_1)
\otimes L_\infty(S_2,\Si_2,\mu_2) $. The same situation appears
whenever two sub-\sif s $ \Si_1, \Si_2 \subset \Si $ are independent
(that is, $ \mu(X \cap Y) = \mu(X) \mu(Y) $ for all $ X \in \Si_1 $, $
Y \in \Si_2 $), similarly to \cite[Sect.~1c]{Ts}.
\end{sloppypar}

\begin{definition}
Let $ (S,\Si,\M) $ be a measure class space. Two sub-\sif s $ \Si_1,
\Si_2 \subset \Si $ are \independent{\M}, if they are
\independent{\mu} for some $ \mu \in \M $.
\end{definition}
 
If $ \Si_1, \Si_2 $ are \independent{\M} then (up to a natural unitary
equivalence) $ L_\infty (\Si_1 \vee \Si_2) = L_\infty (\Si_1) \otimes
L_\infty (\Si_2) $ (as before, $ L_\infty (\Si_1) = L_\infty
(S,\Si_1,\M|_{\Si_1}) $ etc).

The product $ (S,\Si,\M) = (S_1,\Si_1,\M_1) \times (S_2,\Si_2,\M_2) $
of two measure class spaces is a measure class space \cite[14.4]{Ar};
namely, $ (S,\Si) = (S_1,\Si_1) \times (S_2,\Si_2) $, and $ \M $ is
the equivalence class containing $ \mu_1 \times \mu_2 $ for some
(therefore all) $ \mu_1 \in \M_1 $, $ \mu_2 \in \M_2 $. In this case $
L_\infty(S,\Si,\M) = L_\infty(S_1,\Si_1,\M_1) \otimes
L_\infty(S_2,\Si_2,\M_2) $.

Given two commutative von Neumann algebras $ \A_1 $ on $ H_1 $ and $
\A_2 $ on $ H_2 $, their tensor product $ \A = \A_1 \otimes \A_2 $ is
a von Neumann algebra on $ H = H_1 \otimes H_2 $. Given isomorphisms $
\al_1 : \A_1 \to L_\infty(S_1,\Si_1,\M_1) $ and $ \al_2 : \A_2 \to
L_\infty(S_2,\Si_2,\M_2) $, we get an isomorphism $ \al = \al_1
\otimes \al_2 : \A \to L_\infty(S,\Si,\M) $, where $ (S,\Si,\M) =
(S_1,\Si_1,\M_1) \times (S_2,\Si_2,\M_2) $; namely, $ \al ( A_1
\otimes A_2 ) = \al_1 (A_1) \otimes \al_2 (A_2) $ for $ A_1 \in \A_1
$, $ A_2 \in \A_2 $. Note that $ \al ( \A_1 \otimes \One ) = L_\infty
(\ti\Si_1) $ and $ \al ( \One \otimes \A_1 ) = L_\infty (\ti\Si_2) $,
where $ \ti\Si_1 = \{ A_1 \times S_2 : A_1 \in \Si_1 \} $ and $
\ti\Si_2 = \{ S_1 \times A_2 : A_2 \in \Si_2 \} $ are \independent{\M}
sub-\sif s of $ \Si $, and $ \ti\Si_1 \vee \ti\Si_2 = \Si $.

\begin{corollary}\label{2a2}
For every isomorphism $ \al : \A_1 \otimes \A_2 \to L_2(S,\Si,\M) $
there exist \independent{\M} $ \Si_1,\Si_2 \subset \Si $ such that $
\al (\A_1 \otimes \One) = L_\infty(\Si_1) $, $ \al (\One \otimes \A_2)
= L_\infty(\Si_2) $, and $ \Si_1 \vee \Si_2 = \Si $.
\end{corollary}

\subsection{Spectrum as a measure class factorization}
\label{2b}

As before, $ B $ is a noise-type Boolean algebra. The corresponding
projections $ Q_x $ commute (by \eqref{1a1}), and generate a commutative
von Neumann algebra $ \A $. Sect.~\ref{2a} gives us a measure class
space $ (S,\Si,\M) $ and an isomorphism
\[
\al : \A \to L_\infty(S,\Si,\M) \, .
\]
Projections $ Q_x $ turn into indicators:
\[
\al(Q_x) = \One_{S_x} \, , \quad S_x \in \Si
\]
(of course, $ S_x $ is an equivalence class rather than a set);
\eqref{1a1} gives
\begin{equation}\label{2b2}
S_x \cap S_y = S_{x \wedge y} \, .
\end{equation}
(In contrast, the evident inclusion $ S_x \cup S_y \subset S_{x\vee y}
$ is generally strict.)
Every \measurable{\Si} set $ E \subset S $ leads to a subspace $ H(E)
\subset H $ such that
\[
\al \( \operatorname{Pr}_{H(E)} \) = \One_E \, ,
\]
where $ \operatorname{Pr}_{H(E)} $ is the projection onto $ H(E)
$. Note that
\begin{gather*}
H ( E_1 \cap E_2 ) = H(E_1) \cap H(E_2) \, , \\
H ( E_1 \uplus E_2 ) = H(E_1) \oplus H(E_2) \, , \\
H ( E_1 \cup E_2 ) = H(E_1) + H(E_2) \, , \\
H(S_x) = H_x
\end{gather*}
(the second line differs from the third line by assuming that $ E_1,
E_2 $ are disjoint and concluding that $ H(E_1), H(E_2) $ are
orthogonal); $ E \mapsto H(E) $ is a projection measure.

Every subset of $ B $ leads to a subalgebra of $ \A $, thus, to a
sub-\sif\ of $ \Si $. In particular, for every $ x \in B $ we
introduce the von Neumann algebra
\[
\A_x \subset \A \quad \text{generated by} \quad \{ Q_y : y \in B, x
\vee y = 1 \}
\]
and the \sif\ $ \Si_x \subset \Si $ such that
\[
\al(\A_x) = L_\infty(\Si_x) \, .
\]
Note that
\begin{equation}\label{2b3}
x \le y \quad \text{implies} \quad \A_x \subset \A_y \quad \text{and}
\quad \Si_x \subset \Si_y \, .
\end{equation}

The Boolean algebra $ B $ contains the least element $ 0 $; the
corresponding operator $ Q_0 $, --- the projection onto the
one-dimensional space $ H_0 $ (of constant functions on $ (\Om,\F,P)
$), --- is a \emph{minimal} projection in $ \A $. Combined with the
relation $ \al(Q_0) = \One_{S_0} $ it shows that $ S_0 $ is an atom of
$ \Si $. Similarly, $ Q_{x'} $ is a minimal projection in $ \A_x $,
and therefore
\[
S_{x'} \text{ is an atom of } \Si_x \, .
\]

\begin{proposition}
$ \Si_x \vee \Si_y = \Si_{x \vee y} $ for all $ x,y \in B $.
\end{proposition}

\begin{proof}
By \eqref{2b3}, $ \Si_x \vee \Si_y \subset \Si_{x\vee y} $. By
\ref{2a15} it is sufficient to prove that $ \A_{x\vee y} \subset \A_x
\vee \A_y $, that is, $ Q_z \in \A_x \vee \A_y $ whenever $ x \vee y
\vee z = 1 $. We have $ z = ( z \vee x' ) \wedge ( z \vee y' ) $. By
\eqref{1a1}, $ Q_z = Q_{z\vee x'} Q_{z\vee y'} \in \A_x \vee \A_y $,
since $ Q_{z\vee x'} \in \A_x $ and $ Q_{z\vee y'} \in \A_y $.
\end{proof}

\begin{proposition}
If $ x \wedge y = 0 $ then $ \Si_x, \Si_y $ are \independent{\M}.
\end{proposition}

\begin{proof}
It is sufficient to prove that $ \Si_x, \Si_{x'} $ are
\independent{\M} (since $ \Si_y \subset \Si_{x'} $ by \eqref{2b3}).

As was noted before the proof of \ref{1a9}, we have (up to the natural
unitary equivalence) $ H = H_x \otimes H_{x'} $ and $ Q_{u\vee
v} = Q_u^{(x)} \otimes Q_v^{(x')} $ for all $ u,v \in B $ such that
$ u \le x $ and $ v \le x' $.

By \ref{2a2} it is sufficient to prove that all operators of $ \A_x $
are of the form $ A \otimes \One $ (for $ A : H_x \to H_x $), and all
operators of $ \A_{x'} $ are of the form $ \One \otimes B $. We prove
the former; the latter is similar.
If $ z $ satisfies $ x \vee z = 1 $ then $ z = ( z \wedge x ) \vee x'
$ and therefore $ Q_z = Q_{z\wedge x}^{(x)} \otimes Q_{x'}^{(x')} =
Q_{z\wedge x}^{(x)} \otimes \One $, as needed.
\end{proof}

\begin{remark}
The completion of $ B $ (thus, also the closure of $ B $)
determines uniquely the algebra $ \A $ (since $ \A $ is closed in the
strong operator topology) and therefore also the spectral space.
\end{remark}

\subsection{Spectral filters, spectral sets}
\label{2c}

Taking into account that every noise-type Boolean algebra contains a
dense countable Boolean subalgebra and both algebras lead to the same
spectral space, we assume here (in Sect.~\ref{2c}) that $ B $ is a
\emph{countable} noise-type Boolean algebra.

Having only countably many equivalence classes $ S_x $ we may, and
will, treat them as sets (rather than equivalence classes), satisfying
\eqref{2b2} exactly (rather than almost everywhere). Then sets
\[
\Phi_s = \{ x \in B : s \in S_x \} \quad \text{for } s \in S
\]
satisfy $ ( x,y \in \Phi_s ) \equiv ( x \wedge y \in \Phi_s ) $, which
shows that $ \Phi_s $ is either a filter in $ B $ (if $ s \notin S_0
$) or the improper filter, the whole $ B $ (if $ s \in S_0 $). This
way, points of the spectral space may be interpreted as filters on $ B
$ (``spectral filters'').

Every countable Boolean algebra $ B $ is isomorphic to the Boolean
algebra of all clopen (that is, both closed and open) subsets of a
totally disconnected compact metrizable space, so-called Stone space
of $ B $ (homeomorphic to the Cantor set, if $ B $ is atomless).
Filters on $ B $ (maybe improper) correspond bijectively to closed
subsets (maybe empty) of the Stone space. This way, points of the
spectral space may be interpreted as closed subsets of the Stone space
(``spectral sets''), and the relation $ s \in S_x $ holds if and only
if the closed set corresponding to $ s $ is contained in the clopen
set corresponding to $ x $.

\section[Digression: planar spectral sets, etc.]{\raggedright Digression:
  planar spectral sets, etc.}
\label{sec:3}
As noted in \cite[Intro]{Ts}, in the framework of a ``noise as a Boolean
algebra of \sif s'' we consider the \sif s irrespective of the
corresponding domains (in $ \R^n $ or another parameter
space). In contrast, spectral sets defined before \cite[Sect.~9]{Ts04}
for a noise over $ \R $ are compact subsets of $ \R $ (rather than a
Stone space). In this section we return to a parameter space and its
spectral subsets. The parameter space is usually $ \R^n $, but an
arbitrary topological space can be used equally well.

\subsection{Preliminaries: regular open sets}
\label{3a}

Let $ X $ be a topological space. We introduce the set
\[
\Reg(X) = \{ (G,F) : G = \Int(F), \, F = \Cl(G) \}
\]
of all pairs $ (G,F) $ of subsets of $ X $ such that $ G $ is the
interior of $ F $ and at the same time $ F $ is the closure of $ G
$. For $ r \in \Reg(X) $ we denote
\[
G = \Int(r) \, , \quad F = \Cl(r) \, ,
\]
somewhat abusing the symbols ``$\Int$'' and ``$\Cl$'', since $ r $ is
not a subset of $ X $. We introduce on $ \Reg(X) $ a partial order
\[
r \le s \equiv \Int(r) \subset \Int(s) \equiv \Cl(r) \subset \Cl(s)
\]
(the second and third relations being evidently equivalent). It
appears that $ \Reg(X) $ is a Boolean algebra, and
\begin{gather*}
\Int( r \wedge s ) = \Int(r) \cap \Int(s) \, , \\
\Cl( r \vee s ) = \Cl(r) \cup \Cl(s) \, , \\
\Int(r') = X \setminus \Cl(r) \, , \quad \Cl(r') =
 X \setminus \Int(r) \, .
\end{gather*}
Also,
\begin{gather}
\Cl( r \wedge s ) = \Cl \( \Int(r) \cap \Int(s) \) \subset \Cl(r)
 \cap \Cl(s) \, , \\ 
\Int( r \vee s ) = \Int \( \Cl(r) \cup \Cl(s) \) \supset \Int(r)
 \cup \Int(s) \, . \label{3a8}
\end{gather}

For every $ r \in \Reg(X) $ the set $ G = \Int(r) $ is equal to the
interior of its closure; such sets are called regular open. Every
regular open set $ G $ is $ \Int(r) $ for some $ r \in \Reg(X) $,
namely, $ r = (G,\Cl(G)) $. Thus, the Boolean algebra $ \Reg(X) $ is
naturally isomorphic to the Boolean algebra of all regular open
sets. The same holds for regular closed sets.

See \cite[Sect.~4]{Ha}.

\subsection{Back to a topological base}
\label{3b}

Let $ X $ be a topological space, $ A \subset \Reg(X) $ a Boolean
subalgebra, $ B $ a noise-type Boolean algebra, and $ h : A \to B $ a
homomorphism. We are interested in  a map $ F $ from S to the set of
closed subsets of $ X $ such that for every $ a \in A $,
\begin{equation}\label{3b1}
S_{h(a)} = \{ s \in S : F(s) \subset \Cl(a) \} \quad \pmod 0 \, .
\end{equation}
Here are two relevant assumptions.

\begin{assumption}\label{3b2}
There exists a countable subset $ A_0 \subset A $ such that
$ \{ \Int(a) : a \in A_0 \} $ is a (topological) base of $ X $.
\end{assumption}

\begin{assumption}\label{3b3}
For every $ a \in A $ there exist $ a_1, a_2, \dots \in A $ such that
$ a_n \le a_{n+1} $, $ \Cl(a_n) $ is compact,
$ \Cl(a_n) \subset \Int(a) $ for all $ n $, and $ h(a'_n) \downarrow
h(a') $.
\end{assumption}

\begin{lemma}
Assumption \ref{3b2} ensures uniqueness of $ F $
satisfying \eqref{3b1}.
\end{lemma}

\begin{proof}
Every open set is the union of some sets of the base. In
particular,
\begin{equation}\label{3b5}
X \setminus F(s) = \bigcup_{a\in A_0, s\in S_{h(a')}} \Int(a) \, ,
\end{equation}
since $ \Int(a) \subset X \setminus F(s) \equiv F(s) \subset
X \setminus \Int(a) \equiv F(s) \subset \Cl(a') \equiv s \in S_{h(a')}
$.
\end{proof}

\begin{theorem}
Assumptions \ref{3b2}, \ref{3b3} ensure existence of $ F $
satisfying \eqref{3b1}.
\end{theorem}

\begin{proof}
Assumption \ref{3b2} gives us $ A_0 $. We define $ F(\cdot) $
by \eqref{3b5} and prove \eqref{3b1}.

Let $ a \in A $ and $ s \in S_{h(a)} $; we'll prove that $
F(s) \subset \Cl(a) $. To this end it is sufficient to prove that $
X \setminus F(s) \supset \Int(a_0) $ for every $ a_0 \in A_0 $
satisfying $ \Int(a_0) \subset \Int(a') $. We note that $ a_0 \le a'
$, $ a'_0 \ge a $, $ S_{h(a'_0)} \supset S_{h(a)} $, thus $ s \in
S_{h(a'_0)} $. By \eqref{3b5}, $ X \setminus F(s) \supset \Int(a_0)
$.

Let $ a \in A $ and $ F(s) \subset \Cl(a') $; we'll prove that $ s \in
S_{h(a')} $. Assumption \ref{3b3} gives us  $ a_1, a_2, \dots $. It is
sufficient to prove that $ s \in S_{h(a'_n)} $ for every $ n $, since
$ S_{h(a'_n)} \downarrow S_{h(a')} $.

We have $ X \setminus F(s) \supset X \setminus \Cl(a') = \Int(a)
\supset \Cl(a_n) $. Using \eqref{3b5} and compactness of $ \Cl(a_n) $
we find $ b_1,\dots,b_k \in A_0 $ (dependent on $ n $, of course) such
that $ \Int(b_1) \cup \dots \cup \Int(b_k) \supset \Cl(a_n) $ and $
s \in S_{h(b'_1)} \cap \dots \cap S_{h(b'_k)} $. Introducing $ b =
b_1 \vee \dots \vee b_k $ we have $ \Int(b) \supset \Cl(a_n) $
by \eqref{3a8}, and $ s \in S_{h(b')} $ by \eqref{2b2}. Finally, $
S_{h(b')} \subset S_{h(a'_n)} $ since $ b \ge a_n $, and we get $
s \in S_{h(a'_n)} $.
\end{proof}

From now on we assume \ref{3b2} and \ref{3b3}, and consider $ F $
satisfying \eqref{3b1}.

In particular, if $ B $ is countable, $ X $ is the Stone space of $ B
$, $ A $ consists of all clopen sets, and $ h $ is the natural
isomorphism $ A \to B $, then $ F(s) $ is the spectral set in the
sense of Sect.~\ref{2c}.

In general, every monotone sequence in $ B $ converges in $ \Cl(B) $
(the closure of $ B $ in $ \La $, see \cite[Sect.~2a]{Ts}). Thus,
$ \lim_n h(a_n) $ exists in $ \Cl(B) $ for every monotone sequence $
(a_n)_n $ in $ A $.

\begin{proposition}\label{3b7}
The following two conditions on an increasing sequence $ (a_n)_n $ in
$ A $ are equivalent:

(a) $ \lim_n h(a_n) = 1 $;

(b) for almost every $ s $ there exists $ n $ such that $
F(s) \subset \Cl(a_n) $.
\end{proposition}

\begin{proof}
$ h(a_n) \uparrow 1 \equiv S_{h(a_n)} \uparrow S \equiv \ti\forall
s \; \exists n \; s \in S_{h(a_n)} \equiv \ti\forall s \; \exists n \;
F(s) \subset \Cl(a_n) $, where ``$ \ti\forall $'' means ``for almost
all''.
\end{proof}

\begin{corollary}
If there exist $ a_1 \le a_2 \le \dots $ such that $ \lim_n h(a_n) = 1
$ and $ \Cl(a_n) $ is compact for every $ n $, then $ F(s) $ is
compact for almost all $ s $.
\end{corollary}

Given $ r \in \Reg(X) $, $ r \notin A $, we may try to extend $ h $ to
$ r $ by approximation from the inside:
\[
h_-(r) = \sup \{ h(a) : a \in A, \, \Cl(a) \subset \Int(r) \} \, .
\]
Then $ h_-(r) \wedge h_-(r') = 0 $, but the question is,
whether $ h_-(r) \vee h_-(r') = 1 $ or not.

Denote $ \Bd(r) = \Cl(r) \setminus \Int(r) $ (the boundary).

\begin{proposition}\label{3b9}
(a) If $ h_-(r) \vee h_-(r') = 1 $ then $ F(s) \cap \Bd(r)
= \emptyset $ for almost all $ s $.

(b) If $ F(s) \cap \Bd(r) = \emptyset $ for almost all $ s $, and
$ F(s) $ is compact for almost all $ s $, and $ X $ is a regular
topological space, then $ h_-(r) \vee h_-(r') = 1 $ (and therefore $
h_-(r) $ belongs to the noise-type completion of $ B $, and $ \(
h_-(r) \)' = h_-(r') $).
\end{proposition}

\begin{proof}
We have $ h_-(r) \vee h_-(r') = \sup \{ h(a) : a \in A, \, \Cl(a)
\cap \Bd(r) = \emptyset \} = \sup h(a_n) $ for some $ a_n \in A
$ satisfying $ \Cl(a_n) \cap \Bd(r) = \emptyset $ and $ a_1 \le
a_2 \le \dots \, $ Thus, Item (a) follows from \ref{3b7}. We turn to
Item (b). Using Assumption \ref{3b2} and regularity of $ X $ we take $
a_1, a_2, \dots \in A $ such that $ \Cl(a_n) \cap \Bd(r) = \emptyset $
for all $ n $, and $ \cup_n \Int(a_n) = X \setminus \Bd(r)
$. Introducing $ b_n = a_1 \vee \dots \vee a_n $ we have
$ \Cl(b_n) \cap \Bd(r) = \emptyset $ and by \eqref{3a8},
$ \cup_n \Int(b_n) = X \setminus \Bd(r) $. Compactness of $ F(s) $
implies $ F(s) \subset \Int(b_n) $ for some $ n $ (dependent on $ s
$). By \ref{3b7}, $ h(b_n) \uparrow 1 $. On the other hand, $ h(b_n)
\le \sup_k h(a_k) = h_-(r) \vee h_-(r') $.
\end{proof}

\bigskip
\filbreak
{
\small
\begin{sc}
\parindent=0pt\baselineskip=12pt
\parbox{4in}{
Boris Tsirelson\\
School of Mathematics\\
Tel Aviv University\\
Tel Aviv 69978, Israel
\smallskip
\par\quad\href{mailto:tsirel@post.tau.ac.il}{\tt
 mailto:tsirel@post.tau.ac.il}
\par\quad\href{http://www.tau.ac.il/~tsirel/}{\tt
 http://www.tau.ac.il/\textasciitilde tsirel/}
}

\end{sc}
}
\filbreak

\end{document}